\documentclass[11pt]{article}
\usepackage{geometry}
\usepackage{amsmath}
\usepackage{amssymb}
\usepackage{amsthm}
\usepackage{amsfonts}
\usepackage[dvipsnames]{xcolor}

\usepackage[bookmarks,bookmarksnumbered,pdfstartview=FitBH]{hyperref}
\hypersetup{
	colorlinks=true, 
	linkcolor=blue,
	citecolor= blue,
	linktoc= all, 
}

\usepackage{mathtools}

\usepackage{tikz}
\usetikzlibrary{cd}
\usetikzlibrary{arrows.meta}

\DeclareMathOperator{\bs}{\setminus}
\DeclareMathOperator{\la}{\langle}
\DeclareMathOperator{\ra}{\rangle}
\DeclareMathOperator{\Diam}{Diam}
\DeclareMathOperator{\Sub}{Sub}
\DeclareMathOperator{\Stab}{Stab}
\DeclareMathOperator{\h}{\mathfrak{h}}
\DeclareMathOperator{\kk}{\mathfrak{k}}
\DeclareMathOperator{\id}{id}
\DeclareMathOperator{\Comm}{Comm}
\DeclareMathOperator{\Cay}{Cay}
\DeclareMathOperator{\rk}{rk}

\newtheorem{theorem}{Theorem}[section]
\newtheorem{lemma}[theorem]{Lemma}
\newtheorem{proposition}[theorem]{Proposition}
\newtheorem{claim}{Claim}
\newtheorem{corollary}[theorem]{Corollary}

\theoremstyle{definition}
\newtheorem{definition}[theorem]{Definition}
\newtheorem{example}[theorem]{Example}

\theoremstyle{remark}
\newtheorem{remark}[theorem]{Remark}
\newtheorem{question}[theorem]{Question}

\let\OLDthebibliography\thebibliography
\renewcommand\thebibliography[1]{
  \OLDthebibliography{#1}
  \setlength{\parskip}{0pt}
  \setlength{\itemsep}{2pt plus 0.3ex}
}

\newcommand\blfootnote[1]{
  \begingroup
  \renewcommand\thefootnote{}\footnote{#1}
  \addtocounter{footnote}{-1}
  \endgroup
}

\title{On subgroups with narrow Schreier graphs}
\author{Pénélope Azuelos}
\date{\vspace{-5ex}}

\begin{document}

\maketitle

\begin{abstract}
    We study finitely generated pairs of groups $H \leq G$ such that the Schreier graph of $H$ has at least two ends and is \emph{narrow}. Examples of narrow Schreier graphs include those that are quasi-isometric to finitely ended trees or have linear growth. Under this hypothesis, we show that $H$ is a virtual fiber subgroup if and only if $G$ contains infinitely many double cosets of $H$. Along the way, we prove that if a group acts essentially on a finite dimensional CAT(0) cube complex with no facing triples then it virtually surjects onto the integers with kernel commensurable to a hyperplane stabiliser.
\end{abstract}

\section{Introduction}
\blfootnote{\\{\small \textbf{{Keywords:}} Schreier graph, quasi-line, virtual fiber subgroup, separable subgroup, bounded packing, Houghton's groups}

\smallskip
\noindent{\small \textbf{{2020 Mathematics Subject Classification:}} 20F65}}
A well-known result of Hopf \cite{Hopf} states that any finitely generated two-ended group is virtually cyclic. Analogous results exist for pairs of finitely generated groups $H \leq G$ which are two-ended in appropriate ways, subject to $H$ being ``sufficiently normal". The first such result is a theorem of Houghton stating that, if $H$ has two relative ends and infinite index in its normaliser, then there exists  a finite index subgroup $G' \leq G$ such that $H$ is a normal subgroup of $G'$ and $G'/H$ is infinite cyclic \cite[Theorem~3.7]{Houghton}. 

The number of relative ends of a subgroup $H \leq G$ was introduced by Houghton in the same paper. It can be defined as follows. Fix a finite generating set $S$ for $G$ and let $\Cay(G,S)$ be the Cayley graph of $G$ with respect to $S$. Given a subspace $K \subseteq \Cay(G,S)$, we say that a connected component of $\Cay(G,S) - K$ is \textit{deep} if it contains points arbitrarily far from $K$. Given $R \geq 0$, let $\mathcal{N}_R(H)$ denote the $R$-neighbourhood of $H$. The number of \textit{relative ends} of $H$ is the supremum over $R \geq 0$ of the number of $H$-orbits of deep components of $\Cay(G,S) - \mathcal{N}_R(H)$. In contrast, the number of \textit{filtered ends} of $H$ is the supremum over $R \geq 0$ of the number of deep components of $\Cay(G,S) - \mathcal{N}_R(H)$. This invariant was introduced by Kropholler and Roller in \cite{Kropholler-Roller}; a more geometric perspective can be found in \cite[Chapter~14]{Geoghegan}. 

If $H$ has two filtered ends it might not have two relative ends but there is a subgroup of $H$ of index at most 2 which does. This stronger assumption allows for the condition that $H$ has infinite index in its normaliser to be weakened as follows. Two subgroups $H,K \leq G$ are \textit{commensurable} if $H \cap K$ has finite index in both $H$ and $K$. The \textit{commensurator} of $H$ is the subgroup 
\[\Comm_G(H) \coloneqq \{g \in G: H \text{ and } gHg^{-1} \text{ are commensurable}\} \leq G.\]
Clearly $N(H) \leq \Comm_G(H)$, where $N(H)$ denotes the normaliser of $H$. Kropholler and Roller showed that, if $H$ has two filtered ends and infinite index in its commensurator, then there are finite index subgroups $G' \leq G$ and $H' \leq H$ such that $H'$ is a normal subgroup of $G'$ and $G'/H'$ is infinite cyclic \cite[Theorem~1.3]{Kropholler-Roller}. 

We alter the condition that $H$ has two relative ends in a different way. We say that an unbounded connected graph $Y$ is \textit{narrow} if for every $\mu \geq 1$ there exists an upper bound on the number of pairwise disjoint, unbounded, $\mu$-coarsely connected subspaces one can find in $Y$ (see Definition~\ref{narrow defn}). Examples of narrow graphs include graphs with linear growth and infinite uniformly locally finite quasi-trees with finitely many ends (see Proposition~\ref{linear growth implies narrow}). Under the assumption that the quotient $H \bs \Cay(G,S)$ has at least two ends and is narrow for some (equivalently any) finite $G$-generating set $S$, we can weaken the condition on the normaliser to the assumption that there are infinitely many distinct double cosets of $H$ (see Theorem~\ref{main} below).

We say that $H \leq G$: \\
    -- is a \textit{virtual fiber subgroup} if it is finitely generated and there are finite index subgroups $G' \leq G$ and $H' \leq H$ such that $H' \unlhd G'$ and $G'/H' = \mathbb{Z}$; \\
    -- is \textit{separable} if for every $g \in G - H$ there is a finite index subgroup $K \leq G$ such that $H \leq K$ and $g \notin K$, or equivalently if $H$ is closed with respect to the profinite topology on $G$; \\
    -- has \textit{bounded packing} if, having fixed a proper word metric $d_S$ on $G$, for every $D \geq 0$ there is a constant $R_D \geq 0$ such that in every set $\mathcal{A}$ of $R_D + 1$ left cosets of $H$ there are $gH, g'H \in \mathcal{A}$ such that $d_S(gH, g'H) > D$.

\begin{theorem} \label{main}
    Let $G$ be a finitely generated group and suppose that $H \leq G$ is a finitely generated subgroup such that $H \bs G$ is narrow and has at least two ends. Then the following statements are equivalent:
    \begin{enumerate}
        \item $H$ is a virtual fiber subgroup.
        \item $H$ is separable.
        \item $H$ has bounded packing.
        \item There are infinitely many distinct double cosets of $H$.
    \end{enumerate}
\end{theorem}

\begin{remark}
    The descending implications of the theorem all hold if we only assume that $G$ is finitely generated and $[G:H] = \infty$. The implication (2) $\Rightarrow$ (3) is due to Yang \cite{Yang} and the other two are straight-forward. To deduce (1) $\Rightarrow$ (2), suppose that $H$ is a virtual fiber subgroup and let $H' \leq H$ and $G'\leq G$ be finite index subgroups such that there is a surjective homomorphism $\pi: G' \rightarrow \mathbb{Z}$ with kernel $H'$. Then $H' = \cap_{n \in \mathbb{N}} \pi^{-1}(n\mathbb{Z})$ and each $\pi^{-1}(n\mathbb{Z})$ has finite index in $G$. Therefore $H'$ is closed with respect to the profinite topology on $G$ and, since $H$ is a finite union of cosets of $H'$, so is $H$. To see that (3) $\Rightarrow$ (4), note that $d_S(H, hgH) = d_S(H,gH)$ for all $g \in G$ and $h \in H$. Thus if $|H \bs G / H| < \infty$ then $d_S(gH, g'H) = d_S(H, g^{-1} g' H)$ is uniformly bounded and $H$ cannot have bounded packing unless $[G:H] < \infty$.
    
    In contrast, none of the ascending implications hold without some extra condition on $H$ or on $G$ (see Examples~\ref{2not1}, \ref{3not2}, \ref{4not3} below). Moreover, none of the properties (1)-(4) follow directly from the fact that $H \bs G$ is narrow with at least two ends, or even from the stronger assumption that $H\bs G$ is quasi-isometric to a line. This is illustrated in Section~\ref{examples} using a class of groups introduced by Houghton \cite{Houghton78}. Finally, the necessity of the condition that $H$ is finitely generated is illustrated by Example~\ref{Grigorchuk} below.
\end{remark}

    \begin{example} \label{2not1}
        If $G$ is a hyperbolic surface group and $H \leq G$ is cyclic then $H$ is separable by \cite{Scott78} but $H$ is not a virtual fiber subgroup. Note that $H$ also has two relative ends and two filtered ends.  
    \end{example}
    \begin{example} \label{3not2}
    Burger and Mozes constructed in \cite{Burger-Mozes} a family of finitely presented simple groups which split as an amalgamated free product $F \ast_E F$, where $F$ is a free group of finite rank and $E$ is a finite index subgroup of both factors. Let $G = F \ast_E F$ be such a group. Then, since $G$ acts on a locally finite tree with $E$ an edge stabiliser, \cite[Theorem~3.2]{Hruska-Wise} implies that $E$ has bounded packing in $G$. However, the fact that $G$ is simple prevents it from having any proper finite index subgroups, so $E$ cannot be separable.
    \end{example}
    \begin{example} \label{4not3}
    Let $G,H$ be finitely generated groups such that $H \leq G$ and $H$ does not have bounded packing in $G$. Then $H \times \{0\}$ does not have bounded packing in $G \times \mathbb{Z}$ but $\{H \times \{n\} : n \in \mathbb{Z}\}$ are pairwise distinct double cosets of $H$. 
    \end{example}
    \begin{example} \label{Grigorchuk}
    Let $G$ be the Grigorchuk group -- viewed as a group of automorphisms of the 2-regular rooted tree $T$ -- and recall that $G$ is a finitely generated torsion group, so in particular cannot contain any virtual fiber subgroups (see e.g. \cite{Grigorchuk2005}). Given a ray $\xi$ based at the root of $T$, let $G_\xi$ denote the stabiliser of $\xi$. The stabiliser of each vertex of $\xi$ has finite index in $G$ so $G_\xi$ is an intersection of finite index subgroups and is therefore separable. For many choices of $\xi$, the Schreier graph of $G_\xi$ is quasi-isometric to a line (see e.g. \cite[page~201]{Grigorchuk2005}) so in particular it is narrow and two-ended yet $G_\xi$ does not satisfy the conclusion of Theorem~\ref{main}.
    \end{example}

There are many examples of group pairs that satisfy the conditions of Theorem~\ref{main}. For instance if $H \unlhd G$ is a finitely generated pair such that $G / H = \mathbb{Z}$ then $G/H$ is a quasi-line, so it is narrow by Proposition~\ref{linear growth implies narrow} and has two ends. More generally, if $H$ is a virtual fiber subgroup of a finitely generated group $G$ then $H \bs G$ is quasi-isometric to either $\mathbb{Z}$ or $\mathbb{N}$ (see Corollary~\ref{virtual fibers are narrow}) so Theorem~\ref{main} is close to optimal. This raises the following natural question. 

\smallskip

\begin{question}
    Does there exist a finitely generated pair of groups $H \leq G$ such that $H \bs G$ is narrow and one-ended (e.g. a quasi-ray) and $|H \bs G / H| = \infty$ but $H$ is not a virtual fiber subgroup?
\end{question}

\smallskip

The following corollary follows from Theorem~\ref{main} and Proposition~\ref{linear growth implies narrow}.

\begin{corollary}
    Let $H\leq G$ be a finitely generated pair of groups such that $H \bs G$ has at least two ends and either $H \bs G$ has linear growth or $H \bs G$ is quasi-isometric to a finitely ended tree. Then $H$ is a virtual fiber subgroup if and only if $|H \bs G / H| = \infty$.
\end{corollary}

The following is a consequence of Theorem~\ref{main} and Corollary~\ref{virtual fibers are narrow}.

\begin{corollary} \label{narrow with too many ends}
    Let $H \leq G$ be a finitely generated pair of groups and suppose that $H \bs G$ has $3 \leq e <\infty$ ends and is narrow. Then $H$ has finitely many double cosets in $G$.
\end{corollary}

Several classes of groups $G$ have been shown to have the property that if $H \leq G$ is a subgroup such that $|H \bs G / H| < \infty$ then $|G/H| < \infty$. These include finitely generated nilpotent-by-polycyclic groups \cite[Proposition~3.20]{Cornulier2006} and branch groups \cite[Theorem~A]{Francoeur}. We therefore have the following corollary of Theorem~\ref{main}.

\begin{corollary}
    Let $G$ be a finitely generated group and suppose that $G$ is either nilpotent-by-polycyclic or a branch group. If $H \leq G$ is finitely generated and $H \bs G$ is narrow with at least two ends then $H$ is a virtual fiber subgroup.
\end{corollary}

Note that nilpotent-by-polycyclic groups also have the property that every finitely generated subgroup has bounded packing \cite[Theorem~3.7]{Sardar} but they are not in general subgroup separable.

\paragraph{Consequences for the space of subgroups.}
Let $\Sub(G)$ denote the set of all subgroups of $G$ equipped with the topology induced from the Cantor set $\{0,1\}^G$ (viewed as the power set of $G$). In previous work with Damien Gaboriau, we showed that if $G$ is finitely generated and $H \leq G$ is such that $H \bs G$ is quasi-isometric to a line then there are countably many intermediate subgroups $H \leq K \leq G$ \cite[Proposition~3.15]{AG23}. In the case where $H$ is finitely generated this implies that $H$ is not in the perfect kernel of $\Sub(G)$. More precisely, we showed that such a subgroup $H$ vanishes by the third Cantor-Bendixson derivative of $\Sub(G)$ \cite[Theorem~3.10]{AG23} (i.e. $\rk_{CB}^\times(G;H) \leq 3$ in the language of that paper). Combining this with the results of this paper we obtain the following.

\begin{corollary}
    Let $G$ be finitely generated and $H \leq G$ be a subgroup such that $H \bs G$ is narrow. Then there are countably many intermediate subgroups $H \leq K \leq G$. If in addition $H$ is finitely generated and $H \bs G$ has at least two ends then  $\rk_{CB}^\times(G;H) \leq 3$.
\end{corollary}
\begin{proof}
    Let $H \leq K \leq G$ be an intermediate subgroup. By Proposition~\ref{index restrictions} either $[G:K]<\infty$ or $[K:H]<\infty$. For each $k \in \mathbb{N}$, there are finitely many subgroups of $G$ of index $k$ and there are at most countably many subgroups of $G$ containing $H$ with index $k$ so this implies that there are countably many such subgroups $H \leq K \leq G$. 
    Suppose that $H$ is finitely generated and that $H \bs G$ has at least two ends. Theorem~\ref{main} implies that either $H$ is a virtual fiber subgroup or $|H \bs G / H| < \infty$. In the former case Corollary~\ref{virtual fibers are narrow} implies that $H \bs G$ is a quasi-line so  $\rk_{CB}^\times(G;H) \leq 3$ by \cite[Theorem~3.10]{AG23} and in the latter $H$ is isolated in $\Sub(G)$ so  $\rk_{CB}^\times(G;H) =1$.
\end{proof}

\paragraph{Bounded packing.} The bounded packing property was introduced by Hruska and Wise in \cite{Hruska-Wise}. Examples of subgroups which satisfy this property include, on the one hand, subgroups which are normal (this was shown by Hruska and Wise) or commensurated \cite{Connor-Mihalik-2013}, and, on the other, quasiconvex subgroups of hyperbolic groups \cite{Sageev97} (see also \cite[Theorem~2]{Chepoi-Dragan-Vaxes}) and many of their generalisations (e.g. relatively quasiconvex subgroups of suitable relatively hyperbolic groups \cite{Hruska-Wise}, strongly quasiconvex subgroups of any finitely generated group \cite{Tran}, hierarchically quasiconvex subgroups of hierarchically hyperbolic groups \cite{HHP}). This list is highly non-exhaustive.

\paragraph{CAT(0) cube complexes.}
The main content of Theorem~\ref{main} is the implication $(3) \Rightarrow (1)$ and this is proved using an action of $G$ on a finite dimensional CAT(0) cube complex $C$. This cube complex is obtained using Sageev's construction \cite{Sageev95} applied to the codimension one subgroup $H$ and the fact that it is finite dimensional follows from the fact that $H$ has bounded packing \cite{Hruska-Wise}. This will allow us to use tools developed by Caprace and Sageev in \cite{Caprace-Sageev} to show that $C$ has no facing triples of hyperplanes.
We will conclude using the following proposition, which may be of independent interest.

\begin{proposition} \label{no facing triples to action on Z}
Let $\Gamma$ be a group and suppose that $\Gamma$ acts essentially on a finite dimensional CAT(0) cube complex $Y$ which has no facing triples. Then for any hyperplane $\h$ in $Y$, there are finite index subgroups $K  \leq \Stab_\Gamma(\h)$ and $\Lambda \leq \Gamma$ such that $K \unlhd \Lambda$ and $\Lambda / K = \mathbb{Z}$.
\end{proposition}

The fact that such a group virtually surjects onto $\mathbb{Z}$ is not entirely new but the control we have over the kernel appears to be (see Remark~\ref{comparisons}).

\paragraph{Outline of the paper.} We recall some useful definitions and facts relating to CAT(0) cube complexes in Section~\ref{preliminaries}. In Section~\ref{narrow quotients}, we define the notion of a narrow graph and show that graphs with linear growth and graphs that are infinite, uniformly locally finite, finitely ended quasi-trees are narrow. Theorem~\ref{main} is proved in Section~\ref{proof}, modulo Proposition~\ref{no facing triples to action on Z} which is proved in Section~\ref{4}. In Section~\ref{examples}, we present examples for each $n \geq 2$ of finitely generated pairs of groups $H \leq G$ such that $H \bs G$ is an $n$-ended quasi-tree with linear growth (in particular narrow) but $H$ is not a virtual fiber subgroup.

\paragraph{Acknowledgements.}
I am very grateful to Mark Hagen and Indira Chatterji for many helpful comments and suggestions and for carefully reading several drafts of this paper. I thank Yves Cornulier and Paul-Henry Leemann for pointing out Example~\ref{Grigorchuk} and Adrien Le Boudec for asking me whether a previous version of Theorem~\ref{main} could be extended to include subgroups whose Schreier graphs have linear growth. I am grateful to the anonymous referee for a number of helpful comments and corrections. This work was funded by a University of Bristol PhD scholarship.

\section{Cube complexes and facing triples} \label{preliminaries}

We assume some familiarity with actions on CAT(0) cube complexes; introductory material can be found in \cite{Chatterji-Niblo, Hruska-Wise-2014, Nica}. Recall that in \cite{Sageev95} Sageev gave a construction which, given a pair of groups $H \leq G$ such that $H$ has at least two relative ends, produces an unbounded CAT(0) cube complex on which $G$ acts with a single orbit of hyperplanes such that $H$ has finite index in the stabiliser of some hyperplane $\mathfrak{h}$. As was pointed out by Haglund and Wise in \cite[Corollary~3.1]{Haglund-Wise}, it was observed in \cite{Sageev97} that if $H$ is finitely generated and has bounded packing then Sageev's construction yields a finite dimensional cube complex (see also \cite[Corollary~3.31]{Hruska-Wise-2014}). 

We will work with the piecewise $\ell^1$ metric on CAT(0) cube complexes (see \cite{Miesch}) which we denote by $d$. Let $g$ be a cubical automorphism of a CAT(0) cube complex $C$ and recall that $g$ is an isometry with respect to both $d$ and the CAT(0) metric (i.e. the piecewise $\ell^2$ metric). A \textit{combinatorial geodesic} in $C$ is a geodesic path $\ell: [0,L] \rightarrow C^{(1)}$ where $L \in \mathbb{N}$ and $\ell(i) \in C^{(0)}$ for each $i \in \mathbb{Z} \cap [0,L]$.
Haglund showed in \cite[Theorem~6.3]{Haglund} that $g$ either fixes a point in $C$ or admits a combinatorial geodesic axis in the \textit{cubical subdivision} $C'$ of $C$ (see \cite[Definition~2.3, Lemma~4.2]{Haglund}).
In the later case, $g$ is called \textit{hyperbolic} and admits a (not necessarily unique or combinatorial)
geodesic axis in $C$ (i.e. a bi-infinite $\ell^1$-geodesic on which $\la g \ra$ acts by translation). Such an isometry $g$ is also hyperbolic in the usual CAT(0) sense and as such also admits an $\ell^2$ geodesic axis.

Let us recall some notions and results from \cite{Caprace-Sageev}.
Let $C$ be a finite dimensional cube complex and $G$ be a group acting by cubical automorphisms on $C$. 
The action is said to be \textit{essential} if one (equivalently every) $G$-orbit contains points arbitrarily far from any halfspace of $G$. If $\h$ is a hyperplane of $C$ then we denote by $\h^+, \h^-$ the halfspaces bounded by $\h$. For any $g \in G$ we write $g \h^+$ (respectively $g\h^-$) to denote the image of $\h^+$ (respectively $\h^-$) under $g$ and $(g\h)^+$ to denote an independent choice of halfspace bounded by $g\h$. An element $g \in G$ is said to \textit{skewer} $\h$ if there is some $n \in \mathbb{N}$ such that either $g^n \h^+ \subsetneq \h^+$ or $g^n \h^- \subsetneq \h^-$. Caprace and Sageev show in \cite[Lemma~2.3]{Caprace-Sageev} that an element $g \in G$ skewers a hyperplane $\h$ if and only if $g$ is hyperbolic and one (equivalently any) axis of $g$ intersects $\h$ at a single point. This statement refers to geodesic axes of $g$ with respect to the CAT(0) metric but the same proof shows that the statement holds for combinatorial geodesic axes of $g$. The following characterisation of essentiality will be useful.

\begin{proposition}[Caprace-Sageev, {\cite[Lemma 2.4, Proposition 3.2]{Caprace-Sageev}}] \label{CS skewering}
Let $C$ be a finite dimensional CAT(0) cube complex and $G$ be a group acting on $C$ with a single orbit of hyperplanes. Then the following are equivalent.
\begin{itemize}
    \item $C$ is unbounded.
    \item A hyperplane of $C$ is skewered by some element of $G$.
    \item Every hyperplane of $C$ is skewered by some element of $G$.
    \item The action of $G$ on $C$ is essential.
\end{itemize}
\end{proposition}

Suppose that $\h \subseteq C$ is a hyperplane and $g \in G$ is a hyperbolic isometry which does not skewer $\h$. If some (equivalently every) axis of $g$ lies in a neighbourhood of $\h$ then $g$ is said to be \textit{parallel} to $\h$. Otherwise $g$ is \textit{peripheral} to $\h$.

\begin{lemma}[Caprace-Sageev, {\cite[Lemma 4.4]{Caprace-Sageev}}] \label{CS peripheral}
    If $C$ is a finite dimensional CAT(0) cube complex, $\h \subseteq C$ is a hyperplane and $g$ is a hyperbolic isometry of $C$ which is peripheral to $\h$ then there exists $n \in \mathbb{N}$ such that $g^n\h \cap \h = \emptyset$.
\end{lemma}

A \textit{facing triple} is a disjoint triple of hyperplanes $\{\h_0, \h_1, \h_2\}$ such that, for each $i \in \mathbb{Z} / 3\mathbb{Z}$, $\h_i$ does not separate $\h_{i+1}$ from $\h_{i+2}$. The classification of hyperbolic isometries is greatly simplified when $C$ has no facing triples:

\begin{lemma} \label{Euclidian isometries}
Let $C$ be a finite dimensional CAT(0) cube complex and suppose that $C$ has no facing triples. Let $g$ be a hyperbolic isometry of $C$ and $\h$ be a hyperplane. Then either $g$ skewers $\h$ or there is some $n \in \mathbb{N}$ such that $g^n\h = \h$.
\end{lemma}
\begin{proof}
First note that $C$ is proper by \cite[Corollary~3.9]{Hagen}. In particular, by \cite[Lemma~4.3]{Caprace-Sageev}, the isometry $g$ is parallel to $\h$ if and only if there is a power of $g$ which stabilises $\h$.
Moreover, it follows from Lemma~\ref{CS peripheral} that, if $g$ is peripheral to $\h$, then there is some $n \in \mathbb{N}$ such that $\{\h, g^n \h, g^{2n}\h\}$ is a facing triple. Thus if $g$ does not skewer $\h$ then it must be parallel to it.
\end{proof}

\section{Narrow spaces} \label{narrow quotients}

 Given a connected graph $Y$, let $V(Y)$ be the vertex set of $Y$, let $d_Y$ (or $d$ if there is no risk of confusion) denote the path metric on $V(Y)$ where each edge has length 1 and let $B_Y(y,r)$ be the set of vertices with distance at most $r$ from a vertex $y \in V(Y)$. A subset $Z \subseteq V(Y)$ is said to be $\mu$\textit{-coarsely connected} for some $\mu \geq 1$ if, for any $z,z' \in Z$, there exists $\{z_0, \dots, z_n\} \subseteq Z$ such that $z_0 = z$, $z_n = z'$ and $d(z_i,z_{i+1}) \leq \mu$ for each $i \in \{0, \dots, n-1\}$. 

 \begin{definition} \label{narrow defn}
 An unbounded connected graph $Y$ is \textit{narrow} if, for each $\mu \geq 1$, there exists $L(\mu) \geq 1$ such that, if $Y_1, \dots, Y_{L(\mu)+1} \subseteq V(Y)$ are unbounded and $\mu$-coarsely connected, then there exists $i \neq j$ with $Y_i \cap Y_j \neq \emptyset$.
 \end{definition}

\begin{proposition} \label{linear growth implies narrow}
    Let $Y$ be an unbounded connected graph and suppose that at least one of the following holds.
    \begin{itemize}
        \item[(a)] $Y$ has linear growth.
        \item[(b)] $Y$ is uniformly locally finite and quasi-isometric to a finitely ended tree.
    \end{itemize}
    Then $Y$ is narrow.
\end{proposition} 
\begin{proof}
    First suppose that (a) holds and let $C \geq 1$ and $y_0 \in V(Y)$ be such that $|B_Y(y_0,n)| \leq Cn$ for all $n \in \mathbb{N}$. Fix $\mu \geq 1$ and let $L \coloneqq \mu(C+1)$. Suppose that $Y_1, \dots, Y_{L+1} \subseteq V(Y)$ are pairwise disjoint unbounded $\mu$-coarsely connected sets of vertices, ordered such that $d(y_0, Y_1) \leq d(y_0, Y_2) \leq  \dots \leq d(y_0,Y_{L+1})$. For each $i \in \{1, \dots, L+1\}$ and $j \geq d(y_0, Y_i) + \mu$ we have $Y_i \cap (B_Y(y_0, j) - B_Y(y_0,j-\mu)) \neq \emptyset$. Therefore if $n \geq d(y_0, Y_{L+1}) + \mu$ then $|B_Y(y_0,n) - B_Y(y_0,n-\mu)| \geq L+1$. Thus, if $k = d(y_0,Y_{L+1})$, then for any $n \geq k(C+1)$
    \[ |B_Y(y_0,k+n\mu)| \geq n(L+1) \geq Ln + k(C+1) = (C+1)(k+n\mu) \]
    which is a contradiction.

    Now suppose that (b) holds and let $T$ be a finitely ended tree and $f:T \rightarrow Y$ be a $\lambda$-quasi-isometry for some $\lambda \geq 1$. Let $D = 3\lambda^3 + \lambda^2 + 3\lambda$ and let $\overline{f}: T \cup \partial T \rightarrow Y \cup \partial Y$ denote the extension of $f$ to the visual compactification of $T$. We henceforth identify the ends of $T$ with the points in its visual boundary in the natural way. 
    Let $e \in \mathbb{N}$ be the number of ends of $T$ and fix $\mu \geq 1$. Let $R \geq 1$ be such that any ball of radius $D + \mu$ in $Y$ contains at most $R$ vertices and let $L \coloneqq eR$. Let $\xi_1, \dots, \xi_e$ denote the ends of $T$ and note that $\overline{f}(\xi_1), \dots \overline{f}(\xi_e)$ are the ends of $Y$. Suppose that $Y_1, \dots, Y_{L+1} \subseteq Y$ are unbounded $\mu$-coarsely connected subspaces, ordered such that the following holds. There exists $1 = L_0 \leq L_1 \leq \dots \leq L_{e-1} \leq L_e = L+2$ such that, for each $i \in \{1, \dots, e\}$, if $j \in [L_{i-1}, L_i)$, then there is a subset of $Y_j$ which converges towards $\overline{f}(\xi_i)$.
    For each $i \in \{1, \dots, e\}$ let $\gamma_i \subseteq T$ be a geodesic ray which is a representative of $\xi_i$. Up to removing a finite length segment from $\gamma_i$, we can assume that each vertex of $\gamma_i$ separates $T$ into at least two unbounded connected components. It follows that for each vertex $v \in \gamma_i$, the ball of radius $D$ around $f(v)$ separates $Y$ into at least two unbounded connected components. It follows that, if $j \in [L_{i-1},L_i)$, then, for all but finitely many $v \in \gamma_i$, the subspace $Y_j$ contains vertices lying in at least two distinct connected components of $Y - B_Y(f(v),D)$. Therefore there exists a vertex $v_i \in Y$ such that, for all $j \in [L_{i-1},L_i)$, the subspace $Y_j$ contains vertices lying in at least two distinct connected components of $Y - B_Y(v_i,D)$.
    If $B_Y(v_i,D) \cap Y_j = \emptyset$ then there are vertices $x,y \in Y_j$ lying in different connected components of $Y - B_Y(v_i,D)$ such that $d(x,y) \leq \mu$. Let $z \in B_Y(v_i,D)$ be a vertex which lies on a geodesic from $x$ to $y$. Then $d(v_i,x) \leq d(v_i,z) + d(z,x) < D + \mu$. Therefore $Y_j \cap B_Y(v_i, D + \mu) \neq \emptyset$ for all $j \in [L_{i-1}, L_i)$. Since the number of the vertices in $\cup_{i=1}^e B_Y(v_i,D + \mu)$ is at most $L$, there exists $j \neq k$ such that $Y_j \cap Y_k \neq \emptyset$.
\end{proof}

\begin{corollary} \label{virtual fibers are narrow}
    Let $G$ be a finitely generated group and let $H \leq G$ be a virtual fiber subgroup. Let $S$ be any finite generating set of $G$. Then $H \bs \Cay(G,S)$ is quasi-isometric to either $\mathbb{Z}$ or $\mathbb{N}$. In particular $H \bs \Cay(G,S)$ is narrow.
\end{corollary}
\begin{proof}
    Let $H' \leq H$ and $G' \leq G$ be finite index subgroups such that $H' \unlhd G'$ and $G' / H' = \mathbb{Z}$. Let $X_H \coloneqq H \bs \Cay(G,S)$ and $X_{H'} \coloneqq H' \bs \Cay(G,S)$. Then, for any finite generating set $S'$ of $G'$, the graph $X_{H'}$ is quasi-isometric to $H' \bs \Cay(G',S')$, which is a quasi-line. Moreover the quotient map $\rho: X_{H'} \rightarrow X_H$ defined by $\rho(H'g) = Hg$ is a finite degree covering map. Suppose that there exists $K \in \mathbb{N}$ such that $\Diam \rho^{-1}(x) \leq K$ for all $x \in X_H$. Let $x,y \in X_{H'}$ and let $\gamma \subseteq X_H$ be a geodesic from $\rho(x)$ to $\rho(y)$. Let $y' \in X_{H'}$ be the endpoint of the lift of $\gamma$ based at $x$. Then $d_{X_{H'}}(x,y) \leq d_{X_{H'}}(x,y') + d_{X_{H'}}(y',y) \leq d_{X_H}(\rho(x),\rho(y)) + K$. Since $\rho$ is a surjective 1-Lipschitz map, this implies that $\rho$ is a quasi-isometry. If $\Diam \rho^{-1}(x)$ is not uniformly bounded then \cite[Lemma~3.14]{AG23} implies that there is a geodesic ray $\gamma \subseteq X_H$ and two lifts $\gamma_1, \gamma_2 \subseteq X_{H'}$ of $\gamma$ such that the Hausdorff distance between $\gamma_1$ and $\gamma_2$ is infinite. Then $X_{H'}$ is quasi-isometric to the union $\gamma_1 \cup \gamma_2$ (equipped with the induced metric from $X_{H'}$) and it follows that $X_H$ is quasi-isometric to $\gamma$. Proposition~\ref{linear growth implies narrow} then implies that $X_H$ is narrow.
\end{proof}

\begin{proposition} \label{narrow QI}
    Let $Y,Y'$ be unbounded connected graphs. Suppose that $Y'$ is narrow and that there exists a constant $C \geq 1$ and a metrically proper injective map $f:V(Y) \rightarrow V(Y')$ such that $d_{Y'}(f(x),f(y)) \leq C d_Y(x,y) + C$ for all $x,y \in Y$. Then $Y$ is narrow.
\end{proposition}
\begin{proof}
    For each $\mu$ let $L(\mu)$ be such that, if $Y_1', \dots, Y_{L(\mu)+1}' \subseteq V(Y')$ are unbounded and $\mu$-coarsely connected, then there exists $i \neq j$ with $Y_i' \cap Y_j' \neq \emptyset$. Fix $\lambda \geq 1$ and let $Y_1, \dots, Y_{L(C\lambda+C)+1} \subseteq V(Y)$ be unbounded and $\lambda$-coarsely connected. Then $f(Y_1), \dots,$\break$ f(Y_{L(C\lambda+C)+1}) \subseteq V(Y')$ are unbounded and $(C\lambda+C)$-coarsely connected so there exists $i \neq j$ such that $f(Y_i) \cap f(Y_j) \neq \emptyset$. Since $f$ is injective, this implies that $Y_i \cap Y_j \neq \emptyset$.
\end{proof}

  If $S,S'$ are finite generating sets of a group $G$ and $H \leq G$ is a subgroup then the identity map $\id: V(H \bs \Cay(G,S)) \rightarrow V(H \bs \Cay(G,S'))$ is a bijective quasi-isometry. Therefore the above proposition implies that $H \bs \Cay(G,S)$ is narrow if and only if $H \bs \Cay(G,S')$ is narrow. We say that $H\bs G$ is \textit{narrow} if $H \bs \Cay(G,S)$ is narrow for some (equivalently any) finite $G$-generating set $S$.

\begin{proposition} \label{narrow subgroups}
    Let $G$ be a finitely generated group and $H \leq G$ be such that $H \bs G$ is narrow. If $H \leq K \leq G$ is such that $[K:H] < \infty$ then $K \bs G$ is narrow.
\end{proposition}
\begin{proof}
For any finite $G$-generating set $S$ let $X_{H,S} \coloneqq H \bs \Cay(G,S)$ and let $X_{K,S} \coloneqq K \bs \Cay(G,S)$. There is a finite degree covering map $\pi_S: X_{H,S} \rightarrow X_{K,S}$ and in particular any $X_{K,S}$ is unbounded. Fix a finite $G$-generating set $S$. We will show that $X_{K,S}$ is narrow. Let $\mu \geq 1$ and let $S_\mu \coloneqq \{g \in G : |g|_S \leq \mu\}$, where $|.|_S$ denotes word length with respect to $S \cup S^{-1}$. The Schreier graph $X_{H, S_\mu}$ is narrow so there exists $L \geq 1$ such that, for any unbounded 1-connected subsets $Z_1, \dots, Z_{L+1} \subseteq V(X_{H,S_\mu})$, there exists $i\neq j$ such that $Z_i \cap Z_j \neq \emptyset$. Let $Y_1, \dots, Y_{L+1} \subseteq V(X_{K,S})$ be unbounded and $\mu$-coarsely connected. For each $i \in \{1, \dots, L+1\}$, the subgraph of $X_{K,S_\mu}$ induced by $Y_i$ is connected so it admits a spanning tree $T_i$. Let $T_i' \subseteq X_{H,S_\mu}$ be a lift of $T_i$ with respect to $\pi_{S_\mu}$. Then $V(T_1'), \dots, V(T_{L+1}') \subseteq V(X_{H,S_\mu})$ are unbounded and 1-connected subsets of $X_{H,S_\mu}$ so there exists $i \neq j$ such that there exists $x \in V(T_i') \cap V(T_j')$. Thus $\pi_{S_\mu}(x) \in V(T_i) \cap V(T_j)= Y_i \cap Y_j$.
\end{proof}

The following is a generalisation of Proposition~3.15 in \cite{AG23}.

\begin{proposition} \label{index restrictions}
    Let $G$ be a finitely generated group and let $H \leq K \leq G$ be a tower of subgroups such that $H \bs G$ is narrow. Then either $[G:K] < \infty$ or $[K:H] < \infty$.
\end{proposition}

\begin{proof}
    Let $S$ be a finite generating set for $G$, let $X_H \coloneqq H \bs \Cay(G,S)$, $X_K \coloneqq K \bs \Cay(G,S)$ and let $\rho_H: \Cay(G,S) \rightarrow X_H$, $\rho_K: \Cay(G,S) \rightarrow X_K$ denote the quotient maps. Note that $\rho_H$ and $\rho_K$ are covers and that $\rho_K$ factors through a covering map $\pi: X_H \rightarrow X_K$ whose degree is $[K:H]$. Suppose that $[G:K] = \infty$, fix $\overline{x} \in X_K$ and let $\overline{\delta} \subseteq X_K$ be an infinite geodesic ray based at $\overline{x}$. For each point $x \in \pi^{-1}(\overline{x})$ there is a lift of $\overline{\delta}$ based at $x$. These lifts are pairwise disjoint and connected so $|\pi^{-1}(\overline{x})| \leq L(1)$, where $L: \mathbb{N} \rightarrow \mathbb{N}$ is a map which expresses the narrowness of $X_H$ (see Definition~\ref{narrow defn}). Thus $[K:H] \leq L(1) < \infty$.
\end{proof}

\section{Proof of Theorem \ref{main}} \label{proof}

As remarked in the introduction, we only need to prove the ascending implications of the theorem. The implication (4) $\Rightarrow$ (3) is proved in Proposition~\ref{4to3} and (3) $\Rightarrow$ (1) is proved in Proposition~\ref{3to1}. 

\medskip

Let us fix some notation for the rest of this section. If $G$ is a group with a finite generating set $S$, we denote the corresponding word metric on $G$ by $d_S$ and let $X \coloneqq \Cay(G,S)$ be the Cayley graph of $G$ with respect to $S$. If $H \leq G$ is a subgroup, then $X_H \coloneqq H \bs X$ denotes the quotient of $X$ with respect to the left action of $H$ by multiplication (i.e. the Schreier coset graph of $H$) and $\rho_H: X \rightarrow X_H$ denotes the quotient map.

\begin{proposition} \label{4to3}
    Let $G$ be a group with finite generating set $S$ and $H \leq G$ be a finitely generated subgroup such that $X_H$ is narrow.
    Then either the set of double cosets $\{HgH : g \in G\}$ is finite or $H$ has bounded packing in $G$.
\end{proposition}
\begin{proof}
By \cite[Corollary 2.4]{Connor-Mihalik-2013}, an element $g \in G$ is in the commensurator $\Comm_G(H)$ if and only if the cosets $H$ and $gH$ are within finite Hausdorff distance from each other. Thus $g \in \Comm_G(H)$ if and only if $\rho_H(gH)$ is bounded.
Since $H$ is finitely generated, there exists $\mu \geq 1$ such that $\rho_H(gH) \subseteq X_H$ is $\mu$-coarsely connected for any $g \in G$. Since $X_H$ is narrow, there exists $N \in \mathbb{N}$ and $g_1, \dots, g_N \in G$ such that, for all $g \in G - \Comm_G(H)$, there exists $i \in \{1, \dots,N\}$ such that $HgH = Hg_iH$.
If the set $\{HgH:g \in G\}$ is infinite, then this implies that there is an infinite family of double cosets $\{Hk_i H\}_{i \in \mathbb{N}}$ such that $k_i \in \Comm_G(H)$ for each $i \in \mathbb{N}$. Thus $H$ has infinite index in $\Comm_G(H)$ so,  by Proposition~\ref{index restrictions}, $\Comm_G(H)$ has finite index in $G$. By \cite[Lemma~6.1]{Connor-Mihalik-2013}, $H$ has bounded packing in $\Comm_G(H)$ and by \cite[Proposition~2.5]{Hruska-Wise} this implies that $H$ has bounded packing in $G$.
\end{proof}

\begin{proposition} \label{3to1}
    Let $G$ be a group with finite generating set $S$ and $H \leq G$ be a finitely generated subgroup such that $X_H$ is narrow and has at least two ends.
    If $H$ has bounded packing then $H$ is a virtual fiber subgroup.
\end{proposition}
\begin{proof}
    Since $H$ has at least two relative ends and bounded packing, there is an action of $G$ on an unbounded finite dimensional CAT(0) cube complex $C$ with a single orbit of hyperplanes \break $G \cdot \h$ such that $H$ has finite index in the $G$-stabiliser of $\h$, denoted $\Stab_G(\h)$ (see Section~\ref{preliminaries}). By Proposition~\ref{narrow subgroups} and since being a virtual fiber subgroup is closed under commensurability, we can assume without loss of generality that $H = \Stab_G(\h)$. By Proposition~\ref{CS skewering}, the action of $G$ on $C$ is essential and there is a hyperbolic element $\gamma \in G$ which skewers $\mathfrak{h}$. Label the halfspaces of $\h$ and replace $\gamma$ with a proper power if necessary so that $\gamma \h^+ \subsetneq \h^+$.
    No non-zero power of $\gamma$ is contained in $H$ so $\rho_H(\la\gamma\ra)$ is infinite. If $\mu$ is the word length of $\gamma$ with respect to $S \cup S^{-1}$ then $\rho_H(g\la\gamma\ra)$ is $\mu$-coarsely connected for all $g \in G$. Therefore there exists $N \in \mathbb{N}$ and $k_1, \dots, k_N \in G$ such that the $\rho_H$-images of $\{k_i \la \gamma \ra\}$ are infinite and pairwise distinct and, for all $k \in G$, if $\rho_H(k \la \gamma \ra)$ is infinite then $\rho_H(k \la \gamma \ra) = \rho_H(k_i \la \gamma \ra)$ for some $i$.

\begin{claim} \label{bounded image implies intersection}
If $k \in G$ is such that $\rho_H(k \la \gamma \ra)$ is finite then $k \mathfrak{h} \cap \mathfrak{h} \neq \emptyset$.
\end{claim}
\begin{proof}
\renewcommand{\qedsymbol}{$\blacksquare$}
    If $\rho_H(k \la \gamma \ra)$ is finite then $Hk\gamma^n = Hk$ for some $n \in \mathbb{N}$, so $k\gamma^nk^{-1} \in H$. Therefore 
    \[d(\h, k\gamma^n\h) = d(\h, k\h) = d(\h, k\gamma^{-n}\h).\]
    If $k\h \cap \h = \emptyset$ then either $k\h \subseteq \h^+$ or $k\h \subseteq \h^-$. In the former case, either $k\gamma^n\h^+ \subsetneq k\h^+ \subsetneq \h^+$ or $k\gamma^{-n}\h^- \subsetneq k\h^- \subsetneq \h^+$, and in the latter, either $k\gamma^n\h^+ \subsetneq k\h^+ \subsetneq \h^-$ or $k\gamma^{-n} \h^- \subsetneq k\h^- \subsetneq \h^-$. In any case $k\h$ separates $\h$ from either $k\gamma^n \h$ or $k\gamma^{-n}\h$. This implies that either $d(\h, k\h) < d(\h, k\gamma^n\h)$ or $d(\h, k\h) < d(\h, k\gamma^{-n}\h)$, either of which is a contradiction.
\end{proof}

Let $k \in G$ be such that $k\h \cap \h = \emptyset$. Claim~\ref{bounded image implies intersection} implies that $\rho_H(k\la\gamma\ra)$ is infinite, so there is $i \in \{1, \dots, N\}$ such that $Hk\la\gamma\ra = Hk_i\la\gamma\ra$. In particular, there exists $n \in \mathbb{Z}$ and $h \in H$ such that $k = hk_i\gamma^n$. Thus if $\ell$ is an $\ell^2$-geodesic axis for $\gamma$ then $d(\h, k \ell) = d(\h, hk_i\gamma^n\ell) = d(\h, k_i\ell)$. Let $M \coloneqq \max \{d(\h, k_i\ell) : 1 \leq i \leq N\}$.

\begin{claim} \label{no facing triples}
    There is no facing triple of hyperplanes in $C$. 
\end{claim}
\begin{proof}
\renewcommand{\qedsymbol}{$\blacksquare$}
    Suppose that there exist $g, g' \in G$ such that $\{\h, g\h, g'\h\}$ is a facing triple. By \cite[Lemma~2.3]{Caprace-Sageev}, $\ell$ intersects $\h$ at a single point and the complementary components of $\ell$ lie in different halfspaces of $\h$. This implies that $\ell$ is not contained in either $g\h$ or $g'\h$. Therefore if $\ell$ intersects $g\h$ (resp. $g'\h$) then the complementary components of $\ell$ lie in different halfspaces of $g\h$ (resp. $g'\h$). Thus if $\ell$ intersects both $g\h$ and $g'\h$ then $\{\h,g\h,g'\h\}$ cannot be a facing triple. We can therefore assume, up to relabelling $g$ and $g'$ if necessary, that $\ell \cap g\h = \emptyset$. This implies that $\h$ and $\ell$ are both contained in a halfspace bounded by $g\h$, which we denote by $(g\h)^+$. Let $\psi = g\gamma g^{-1}$ and note that $\psi$ skewers $g\h$. Let $n \in \mathbb{Z}$ be such that $\psi^n(g\h)^- \subsetneq (g\h)^-$. Then for all $i \in \mathbb{N}$,
    \[d(\h, g^{-1}\psi^{-in}\ell) = d(\psi^{in}g\h, \ell) \geq i.\]
    Moreover $g^{-1}\psi^{-in}\h \cap \h = g^{-1}\psi^{-in}(\h \cap \psi^{in} g\h) = \emptyset$ since $g\h$ separates $\psi^{in} g\h$ from $\h$. But then $d(\h, g^{-1}\psi^{-in}\ell) \leq M$ for all $i$, which is a contradiction.
\end{proof}

The result then follows from Proposition~\ref{no facing triples to action on Z}.
\end{proof}

\section{Reducing a Euclidean cube complex to a line} \label{4}

The aim of this section is to prove Proposition~\ref{no facing triples to action on Z}.

\medskip

Let $Y$ be a finite dimensional CAT(0) cube complex. For any subspace $A \subseteq Y$ let $\mathcal{H}(A)$ denote the set of hyperplanes of $Y$ which separate points in $A$. Let $d_{Haus}$ denote the Hausdorff distance in $Y$ with respect to the $\ell^1$ metric $d$.

\begin{lemma} \label{finite Hausdorff distance}
    Suppose that $\Gamma$ and $Y$ are as in Proposition~\ref{no facing triples to action on Z}.
    Let $\h \subseteq Y$ be a hyperplane and $\gamma \in \Gamma$ be such that $\gamma \h^+ \subsetneq \h^+$. Then every hyperplane in the symmetric difference $\mathcal{H}(\h) \Delta \mathcal{H}(\gamma \h)$ is skewered by $\gamma$. It follows that $|\mathcal{H}(\h) \Delta \mathcal{H}(\gamma\h)| < \infty$ and $d_{Haus}(\h,\gamma\h) < \infty$. 
\end{lemma}
\begin{proof}
Let $\kk$ be a hyperplane which intersects $\h$ but not $\gamma \h$. If $\kk$ separates $\gamma \h$ from $\gamma^{-n}\h$ for some $n \in \mathbb{N}$ then $\gamma$ skewers $\kk$. So suppose that $\kk$ does not separate $\gamma\h$ from any $\gamma^{-n}\h$. Since there are no facing triples, $\kk$ must intersect $\gamma^{-n}\h$ for all $n \in \mathbb{N}$. Therefore for all $n \geq 1$, $\gamma^n \kk$ intersects both $\h$ and $\gamma\h$ while $\gamma^{-n}\kk$ intersects neither $\h$ nor $\gamma \h$. Similarly, for any $\kk \in \mathcal{H}(\h) \Delta \mathcal{H}(\gamma\h)$, either $\gamma$ skewers $\kk$ or for all $n \neq 0$ we have $\gamma^n \kk \notin \mathcal{H}(\h) \Delta \mathcal{H}(\gamma\h)$.
This implies in particular that any $\kk \in \mathcal{H}(\h) \Delta \mathcal{H}(\gamma\h)$ is not stabilised by any power of $\gamma$ and, by Lemma~\ref{Euclidian isometries}, must therefore be skewered by $\gamma$.

Let $\ell$ be a geodesic axis of $\gamma$ and $\sigma \subseteq \ell$ be a fundamental domain for the action of $\la\gamma\ra$ on $\ell$. Every hyperplane which is skewered by $\gamma$ intersects $\ell$ and is therefore of the form $\gamma^n \kk$, where $\kk$ intersects $\sigma$. At most one hyperplane in each $\la\gamma\ra$-orbit is in $\mathcal{H}(\h) \Delta \mathcal{H}(\gamma\h)$ so this implies that $|\mathcal{H}(\h) \Delta \mathcal{H}(\gamma\h)| \leq |\sigma| < \infty$, where $|\sigma|$ is the length of $\sigma$.

It follows from \cite[Lemma~2.24]{CFI} that $d_{Haus}(\h, \gamma\h) \leq |\mathcal{H}(\h) \Delta \mathcal{H}(\gamma \h)| + d(\h,\gamma\h)$ therefore $d_{Haus}(\h, \gamma\h) < \infty$.
\end{proof}

\begin{lemma} \label{universal skewer}
    Suppose that $\Gamma$ and $Y$ are as in Proposition~\ref{no facing triples to action on Z}.
    Then there is a hyperbolic element $\gamma \in \Gamma$ which skewers every hyperplane of $Y$. 
\end{lemma}
\begin{proof}
Let $\mathcal{H}$ be a (necessarily finite) set of pairwise intersecting hyperplanes of $Y$ and $g \in \Gamma$ be an isometry which skewers each $\h \in \mathcal{H}$. Label the halfspaces and replace $g$ with a proper power if necessary so that $g \h^+ \subsetneq \h^+$ for each $\h \in \mathcal{H}$.
Suppose that there exists a hyperplane $\kk$ of $Y$ which is not skewered by $g$. By \cite[Proposition~3.2]{Caprace-Sageev}, every hyperplane of $Y$ is skewered by an element of $\Gamma$ so there exists $k \in \Gamma$ be such that $k \kk^+ \subsetneq \kk^+$. By Lemma~\ref{Euclidian isometries}, there exists $n_1 \in \mathbb{N}$ such that $g^{n_1}\kk^+ = \kk^+$. 
Moreover, there exists $m \in \mathbb{N}$ such that, for each $\h \in \mathcal{H}$, either $k^m \h^+ = \h^+$, $k^m \h^+ \subsetneq \h^+$ or $k^m \h^- \subsetneq \h^-$. In the first two cases, $g^{n_1n}k^m \h^+ \subseteq g^{n_1n}\h^+ \subsetneq \h^+$ for all $n \in \mathbb{N}$. In the latter case, note that $d_{Haus}(\h, k^m\h) < \infty$ by Lemma~\ref{finite Hausdorff distance} and, for any $n \in \mathbb{N}$ such that $d(g^{nn_1}\h, \h) > d_{Haus}(\h, k^m\h)$, we have that $g^{n_1n}k^m\h^+ \subsetneq \h^+$. Since there are finitely many hyperplanes in $\mathcal{H}$, we can choose $n_2 \in \mathbb{N}$ sufficiently large so that $g^{n_1n_2} k^m \h^+ \subsetneq \h^+$ for all $\h \in \mathcal{H}$. 
We also have that $g^{n_1n_2} k^m \kk^+ \subsetneq g^{n_1n_2} \kk^+ = \kk^+$ so the element $g' \coloneqq g^{n_1n_2} k^m$ skewers every hyperplane in $\mathcal{H}' \coloneqq \mathcal{H} \cup \{\kk\}$. Moreover, for each $\h \in \mathcal{H}$, $\kk$ cannot separate hyperplanes in $\la g \ra \cdot \h$ without being skewered by $g$. By the lack of facing triples, this implies that $\kk$ must intersect $g^n\h$ for some $n$. It then follows from Lemma~\ref{finite Hausdorff distance} that $\kk$ intersects $g^n\h$ for all $n$. In particular, $\kk$ intersects every hyperplane in $\mathcal{H}$, so $\mathcal{H}'$ is pairwise intersecting and $|\mathcal{H}'| \leq \dim(Y)$. If there exists a hyperplane $\kk'$ which is not skewered by $g'$, then the above argument applied to $\mathcal{H}', g'$ and $\kk'$ implies that $\mathcal{H}'' \coloneqq \mathcal{H}' \cup \{\kk'\}$ is pairwise intersecting and there exists $g'' \in \Gamma$ which skewers every hyperplane in $\mathcal{H}''$. We can only iterate this process at most $\dim(Y)-1$ times until we obtain the desired element $\gamma$.
\end{proof}

\begin{proof}[Proof of Proposition \ref{no facing triples to action on Z}]
 If $K$ is a hyperplane stabiliser in the cubical subdivision of $Y$ then there is a hyperplane $\h$ in $H$ such that $[\Stab_\Gamma(\h): K] \leq 2$. Therefore, by replacing $Y$ with its cubical subdivision if necessary, we can assume that $\Gamma$ acts on $Y$ without hyperplane inversions.

By Lemma~\ref{universal skewer} there is a hyperbolic isometry $\gamma \in \Gamma$ which skewers every hyperplane of $Y$. There are finitely many $\la\gamma\ra$-orbits of hyperplanes in $Y$ so we can replace $\gamma$ with a proper power so that $\gamma \h^+ \subsetneq \h^+$ for each hyperplane $\h \subseteq Y$.
Fix a hyperplane $\h \subseteq Y$, let $K \coloneqq \Stab_\Gamma(\h) = \Stab_\Gamma(\h^+)$ and let $\Gamma' \coloneqq \la K, \gamma \ra \leq \Gamma$. 

Let $M^+$ (respectively $M^-$) be the set of hyperplanes of $Y$ which lie in $\h^+$ (respectively $\h^-$). Then $kM^+ = M^+$ for any $k \in K$. Moreover, $\gamma M^+ \subseteq M^+$ and, if $\kk \in M^+ - \gamma M^+$, then $\kk$ either separates $\gamma \h$ from $\h$ or $\kk$ intersects $\gamma\h$ but not $\h$, so $|M^+ - \gamma M^+| < \infty$ by Lemma~\ref{finite Hausdorff distance} and thus $|M^+ \Delta \gamma M^+| < \infty$. Similarly, $|M^+ \Delta \gamma^{-1}M^+| = |\gamma^{-1}M^+ - M^+| < \infty$. If $N$ is a set of hyperplanes such that $|M^+ \Delta N| < \infty$ then $|M^+ \Delta \gamma^{\pm 1} N| \leq |M^+ \Delta \gamma^{\pm 1} M^+| + |\gamma^{\pm 1}(M^+ \Delta N)| < \infty$ and, if $k \in K$, then $|M^+ \Delta kN| \leq |M^+ \Delta kM^+| + |k(M^+\Delta N)| < \infty$. It follows that $|M^+ \Delta gM^+| < \infty$ for all $g \in \Gamma'$, so $M^+$ is \textit{commensurated by the action of }$\Gamma'$. A symmetric argument shows that $M^-$ is also commensurated by the action of $\Gamma'$. This implies that the map $tr_{M^+}: \Gamma' \rightarrow \mathbb{Z}$ given by 
\[ tr_{M^+}(g) = |M^+ - g^{-1}M^+| - |g^{-1}M^+ - M^+|\]
is well defined and, by \cite[Proposition~4.H.1]{Cornulier}, it is a homomorphism. Since $tr_{M^+}(\gamma) < 0$, this homomorphism is non-trivial.

Clearly $K \leq \ker(tr_{M^+})$. Let us show that $K$ has finite index in $\ker(tr_{M^+})$. Given a hyperplane $\kk$ of $Y$, define $L_{\kk} \in \mathbb{Z}$ as follows. If $\kk$ does not separate $\gamma^{n_1}\h$ from $\gamma^{n_2}\h$  for any $n_1,n_2 \in \mathbb{Z}$ then $L_{\kk} \coloneqq 0$. Otherwise
\[L_{\kk} \coloneqq \min\{|n_1 - n_2| : n_1,n_2 \in \mathbb{Z} \text{ and } \kk \text{ separates } \gamma^{n_1}\h \text{ from } \gamma^{n_2}\h\}.\]
It is clear from the definition that $L_{\gamma^n \kk} = L_{\kk}$ for all $n \in \mathbb{Z}$. Therefore, since there are finitely many $\la\gamma\ra$-orbits of hyperplanes in $Y$, there exists $L \geq 0$ such that $L_{\kk} \leq L$ for all $\kk$.
Let $g \in \ker(tr_{M^+})$.
Then $g\h$ intersects $\h$ since otherwise either $|{M^+} - g^{-1}{M^+}| > 0 = |g^{-1}{M^+} - {M^+}|$ or $|g^{-1}{M^+} - {M^+}| = 0 < |{M^+} - g^{-1}{M^+}|$. Moreover, any hyperplane which intersects $g\h$ but not $\h$ is contained in $(M^+ - gM^+) \cup (M^- - gM^-)$. Both $M^+$ and $M^-$ are commensurated by $g$, so these sets are finite which implies that there exist $n_1 < 0 < n_2$ such that $g\h$ does not intersect $\gamma^{n_1}\h$ or $\gamma^{n_2}\h$. Since $\{\gamma^{n_1}\h, g\h, \gamma^{n_2}\h\}$ is not a facing triple, this implies that $g\h$ separates $\gamma^{n_1}\h$ from $\gamma^{n_2}\h$. Hence $0 < L_{g\h} \leq L$ so the $\ker (tr_{M^+})$-orbit of $\h$ is contained in the finite set of hyperplanes which separate $\gamma^{1-L}\h$ from $\gamma^{L-1}\h$. Thus $K$ has finite index in $\ker(tr_{M^+})$. It follows that $K$ is normalised by a finite index subgroup $\Lambda$ of $\Gamma'$ and $\Lambda / K = \mathbb{Z}$.

It remains to show that $\Gamma'$ has finite index in $\Gamma$.
Let $R \coloneqq d_{Haus}(\h, \gamma\h) < \infty$, let $\ell$ be a combinatorial geodesic axis for $\gamma$ and let $\sigma \subseteq \ell$ be a segment of length $R$. Let $1 \leq R' \leq R$ be such that exactly $R'$ of the edges of $\sigma$ are dual to hyperplanes in the $\Gamma$-orbit of $\h$ and let $g_1, \dots, g_{R'} \in \Gamma$ be such that each hyperplane in $\Gamma \cdot \h$ which intersects $\sigma$ is of the form $g_i\h$ for some $i$. Let $g \in \Gamma$. Since $\gamma$ skewers every hyperplane in $Y$, the combinatorial axis $\ell$ intersects $g\gamma^n \h$ for all $n \in \mathbb{Z}$. Therefore the complement $\ell - \cup\{\ell \cap g \gamma^n \h : n \in \mathbb{Z}\}$ is a disjoint union of line segments of length at most $R$ which implies that, for some $n \in \mathbb{Z}$, the hyperplane $g \gamma^n \h$ intersects $\sigma$. Therefore there exists $i \in \{1, \dots, R'\}$ such that $g\gamma^n\h = g_i\h$ which implies that $g_i^{-1}g\gamma^n \in K$ and $g_i^{-1}g \in K \gamma^{-n} \subseteq \Gamma'$. Hence $\Gamma'$ has index at most $R'$ in $\Gamma$.
\end{proof}

\begin{remark}
    If we assume, as in the proof of Proposition~\ref{3to1}, that $\Gamma$ and $H \coloneqq \Stab_\Gamma(\h)$ are finitely generated and $H \bs \Gamma$ is narrow, then there is a much shorter proof:
    
    For each $n \in \mathbb{N}$ let $\mathcal{H}_n$ denote the set of hyperplanes whose distance from $\h$ is exactly $n$. Then the hyperplanes of $\mathcal{H}_n$ contained in $\h^-$ are pairwise intersecting, as are those contained in $\h^+$. Therefore $|\mathcal{H}_n| \leq 2\dim(Y)$ so the stabiliser of any hyperplane in $\mathcal{H}_n$ with respect to the action of $H$ on $\mathcal{H}_n$ has index at most $(2\dim(Y))!$ in $H$. Since $H$ is finitely generated, there are finitely many such subgroups. By \cite[Proposittion~3.2]{Caprace-Sageev}, there exists $g \in \Gamma$ such that $g \h^+ \subsetneq \h^+$. The subgroup $K_1 \coloneqq \cap_{n \in \mathbb{Z}} \Stab_H(g^n \h) \leq H$ is a finite intersection of subgroups of index at most $(2\dim(Y))!$ so it is a finite index subgroup of $H$. Let $K \leq K_1$ be a finite index normal subgroup of $H$. Then for each $n \in \mathbb{Z}$, the index of $K^{g^n}$ in $K_1^{g^n} = K_1$ is equal to the index of $K$ in $K_1$ so the orbit $\{K^{g^n} : n \in \mathbb{Z}\}$ is finite. Therefore there exists $m \in \mathbb{N}$ such that $\la g^m \ra \leq N(K)$. For all $n \in \mathbb{N}$ and $k \in K$, $d(\h, kg^n \h) = d(\h,g^n\h) > d(\h,g^{n-1}\h)$ so each element of $\la g^m \ra$ is in a different coset of $K$ and $K$ has infinite index in its normaliser. Since $H \leq N(K)$ and $H \bs \Gamma$ is narrow, Proposition~\ref{index restrictions} then implies that $[\Gamma:N(K)] < \infty$.
    
    The quotient group $N(K)/K$, equipped with any proper word metric, is quasi-isometric to $K \bs \Gamma$, which is a finite index cover of $H \bs \Gamma$. It is clear from the definition of $K$ that $K \leq \Stab_\Gamma(\h^+)$ so $K \bs \Gamma$ is also a finite index cover of $\Stab_\Gamma(\h^+) \bs \Gamma$. By \cite[\S 3.3]{Sageev95}, $\Stab_\Gamma(\h^+) \bs \Gamma$ has at least two ends and, since it is narrow, $H \bs \Gamma$ has finitely many ends. It follows that $N(K) / K$ has $2 \leq e < \infty$ ends but, since $N(K) / K$ is a group, $e$ must equal $2$ and $N(K)/K$ must be virtually cyclic \cite{Hopf}.
\end{remark}
\smallskip
\begin{remark} \label{comparisons}
    While there are statements in the literature of a similar nature to Proposition~\ref{no facing triples to action on Z}, these do not, to the author's knowledge, directly imply this result. 

     If, in addition to the conditions of the proposition, one assumes for example that $\Gamma$ acts on $Y$ without a fixed point in the visual boundary then \cite[Theorem~E]{Caprace-Sageev} implies that $\Gamma$ fixes a Euclidean flat in $Y$. The task of finding an appropriate map to $\mathbb{Z}$ then requires some work (and probably some CAT(0) geometry). It is moreover possible to reduce to the case where either there is no fixed point in the visual boundary or there is a fixed point in the Roller boundary \cite[Proposition~2.26]{CFI}. In the later case, Caprace's appendix in \cite{CFI} provides a number of epimorphisms from a finite index subgroup of $\Gamma$ to $\mathbb{Z}$ (in fact, in either case one can show that $\Gamma$ virtually fixes the simplicial boundary of $Y$ pointwise which suffices to apply Caprace's argument). The fact that one of these has a kernel with the required properties is not immediate, but it can be shown following much the same argument as the proof of Proposition~\ref{no facing triples to action on Z}.
\end{remark}

    \section{A family of examples} \label{examples}
 
 A family of groups introduced by Houghton \cite{Houghton78} provides an example, for each $n \geq 2$, of a finitely generated pair of groups $H_n \leq G_n$ such that $H_n \bs G_n$ is an $n$-ended quasi-tree with linear growth -- and is therefore narrow by Proposition~\ref{linear growth implies narrow} -- but $H_n$ is not a virtual fiber subgroup. These groups are known to provide examples of pairs of groups with an arbitrary number of relative ends \cite[Example~2.1]{Scott77} and, as will become apparent, the fact that $H_n \bs G_n$ is not only $n$-ended but a quasi-tree with linear growth is obvious from the construction. 

\begin{definition}[Houghton's groups]
    Let $n \in \mathbb{N}$ and $X_n$ be a disjoint union of $n$  copies of $\mathbb{N}$. We will use the notation $X_n = \{1, \dots, n\} \times \mathbb{N}$. A bijection $f:X_n \rightarrow X_n$ is a \textit{translation at infinity} of $X_n$ if there exists a finite subset $K \subseteq X_n$ and constants $a_1, \dots, a_n \in \mathbb{Z}$ such that $f(i,m) = (i,m + a_i)$ for all $(i,m) \in X_n - K$.
    Houghton's group $\mathcal{G}_n$ is the group of all translations at infinity of $X_n$.
\end{definition}

When $n = 1$, $\mathcal{G}_n$ is the group of finitely supported permutations on $\mathbb{N}$, which is not finitely generated. However $\mathcal{G}_n$ is finitely generated for all $n \geq 2$. For $n \geq 3$, Houghton defined these groups via the generating sets $S_n$ defined below. The equivalence of the two definitions essentially follows from the observation of Wiegold \cite{Wiegold} that the commutator subgroup $[\la S_n\ra, \la S_n\ra]$ is the group of finitely supported permutations of $X_n$ (see also \cite[Lemma~2.7]{Lee}) and is implicit in \cite{Brown}. The generating sets are defined as follows.

Let $X$ be the increasing union $\cup_{n \in \mathbb{N}} X_n = \mathbb{N} \times \mathbb{N}$ and note that each $\mathcal{G}_n$ is a subgroup of the permutation group of $X$. For each $i \in \mathbb{N}$, let $g_i: X \rightarrow X$ be defined by 
\[g_i(j,m) =
\begin{cases}
    (j,m-1) \quad &\text{if } j=1, m >1 \\
    (i+1,1) \quad &\text{if } j=1, m=1 \\
    (j,m+1) \quad &\text{if } j=i+1 \\
    (j,m) \quad &\text{otherwise.}
\end{cases}
\]
In other words, $g_i$ is a translation of length 1 along the ``line" $\{1\} \times \mathbb{N} \cup \{i+1\} \times \mathbb{N}$. Whenever $i < n$, $g_i(X_n) = X_n$ and $g_i \in \mathcal{G}_n$.
When $n \geq 3$, $\mathcal{G}_n$ is generated by $S_n \coloneqq \{g_1, \dots, g_{n-1}\}$. When $n=2$, we need to add a transposition. Let $\beta: X \rightarrow X$ be defined by 
 \[\beta(i,m) = 
    \begin{cases}
        (2,1) \quad &\text{if } i=1, m=1 \\
        (1,1)\quad &\text{if } i=2, m=1 \\
        (i,m) \quad &\text{otherwise}.
    \end{cases} \]
Then $\mathcal{G}_2$ is generated by $S_2 \coloneqq \{g_1, \beta\}$ (see \cite[pages 5-6]{Lee}).

\medskip
 
 For each $n \geq 2$, consider the natural right action of $\mathcal{G}_n$ on $X_n$. The Schreier graph $Y_n$ of this action with respect to $S_n$, depicted in Figure~\ref{H2} for $n=2$ and Figure~\ref{H3} for $n=3$, is a quasi-tree with $n$ ends. The Schreier graph of a right action is also the Schreier coset graph of any point stabiliser (up to choice of basepoint). Therefore, if $H_n$ is the $\mathcal{G}_n$-stabiliser of $(1,1) \in X_n$, then $H_n \bs \mathcal{G}_n$ is a quasi-tree with $n$ ends. Moreover, $H_n$ is the group of translations at infinity of $X_n$ which fix $(1,1)$, which is precisely the group of translations at infinity of the disjoint union of $n$ rays: $\{1\} \times \mathbb{N}_{\geq2} \sqcup \{2,\dots,n\}\times\mathbb{N}$. Therefore $H_n$ is abstractly isomorphic to $\mathcal{G}_n$ and is in particular finitely generated. Moreover $\mathcal{G}_n = H_n \cup H_ng_1 H_n$.
 \medskip

\begin{figure}
    \centering
    \begin{tikzpicture}
        \draw[color=Bittersweet] (1.2,0) -- (12.8,0);
        \draw[color=Bittersweet, dashed] (0.5,0) -- (1.2,0);
        \draw[color=Bittersweet, dashed] (12.8,0) -- (13.5,0);
        \draw[color=CadetBlue] (2,0.5) circle (0.5);
        \draw[color=CadetBlue] (4,0.5) circle (0.5);
        \draw[color=CadetBlue] (10,0.5) circle (0.5);
        \draw[color=CadetBlue] (12,0.5) circle (0.5);
        \draw[color=CadetBlue] (7,0) ellipse (1 and 0.7);
        \draw[color=Bittersweet, -Stealth] (2.99,0) -- (3,0);
        \draw[color=Bittersweet, -Stealth] (4.99,0) -- (5,0);
        \draw[color=Bittersweet, -Stealth] (6.99,0) -- (7,0);
        \draw[color=Bittersweet, -Stealth] (8.99,0) -- (9,0);
        \draw[color=Bittersweet, -Stealth] (10.99,0) -- (11,0);
        
        \draw[color=CadetBlue, -{Triangle[open]}] (2.04,1) -- (2.05,1);
        \draw[color=CadetBlue, -{Triangle[open]}] (4.04,1) -- (4.05,1);
        \draw[color=CadetBlue, -{Triangle[open]}] (10.04,1) -- (10.05,1);
        \draw[color=CadetBlue, -{Triangle[open]}] (12.04,1) -- (12.05,1);
        \draw[color=CadetBlue, -{Triangle[open]}] (7.04,0.7) -- (7.05,0.7);
        \draw[color=CadetBlue, -{Triangle[open]}] (6.96,-0.7) -- (6.95,-0.7);
        \filldraw[black] (2,0) circle (1.5pt) node[anchor=north]{\small$(1,3)$};
        \filldraw[black] (4,0) circle (1.5pt) node[anchor=north]{\small$(1,2)$};
        \filldraw[black] (6,0) circle (1.5pt);
        \node[anchor=north] at (5.65,0) {\small$(1,1)$}; 
        \filldraw[black] (8,0) circle (1.5pt);
        \filldraw[black] (10,0) circle (1.5pt) node[anchor=north]{\small$(2,2)$};
        \filldraw[black] (12,0) circle (1.5pt) node[anchor=north] at (12.2,0) {\small$(2,3)$};
        \node[anchor=north] at (8.35,0) {\small$(2,1)$};
        \node[color=Bittersweet, anchor=south] at (5,0.05) {$g_1$};
        \node[color=CadetBlue, anchor=south] at (7,0.75) {$\beta$};
    \end{tikzpicture}
    \caption{Schreier graph of the action of $\mathcal{G}_2$ on $X_2$.}
    \label{H2}
\end{figure}
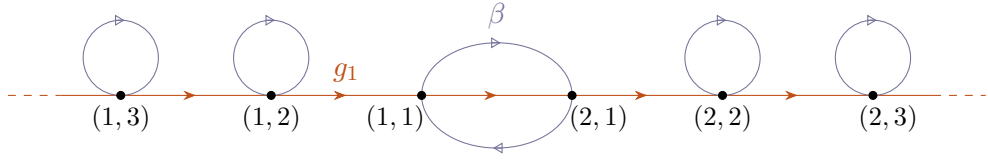

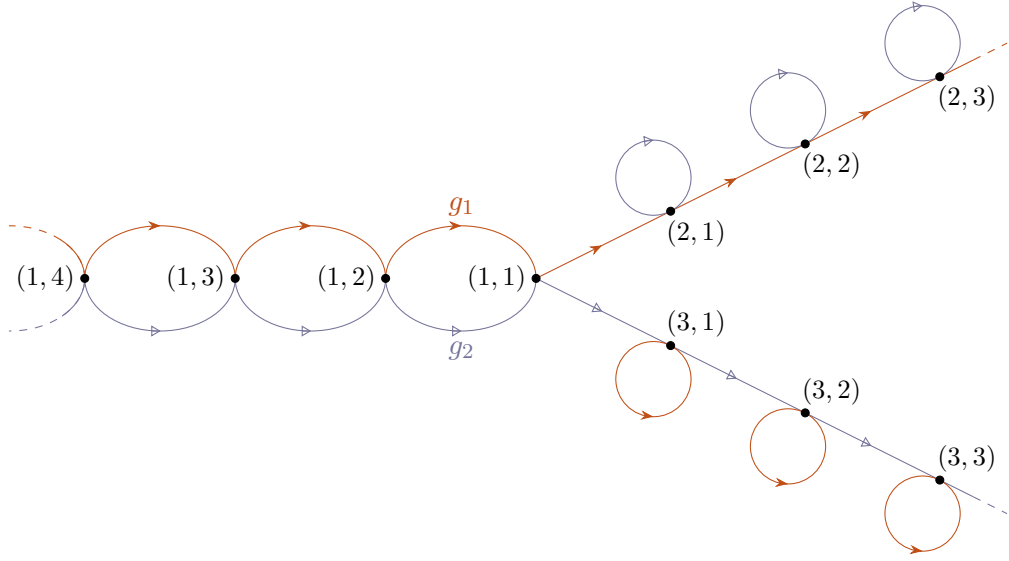
\begin{figure}
    \centering
    \begin{tikzpicture}
        \draw[color=Bittersweet] (6,0) -- (11.814,2.9065);
        \draw[color=Bittersweet, dashed] (11.814,2.9065) -- (12.261,3.13);
        \draw[color=Bittersweet, -Stealth] (6.86,0.43) -- (6.88,0.44);
        \draw[color=Bittersweet, -Stealth] (8.649,1.324) -- (8.669,1.334);
        \draw[color=Bittersweet, -Stealth] (10.438,2.219) -- (10.458,2.229);
        
        \draw[color=CadetBlue] (6,0) -- (11.814,-2.9065);
        \draw[color=CadetBlue, dashed] (11.814,-2.9065) -- (12.261,-3.13);
        \draw[color=CadetBlue, -{Triangle[open]}] (6.86,-0.43) -- (6.88,-0.44);
        \draw[color=CadetBlue, -{Triangle[open]}] (8.649,-1.324) -- (8.669,-1.334);
        \draw[color=CadetBlue, -{Triangle[open]}] (10.438,-2.219) -- (10.458,-2.229);
        
        \draw[color=Bittersweet] (2,0) arc (0:180:1 and 0.7);
        \draw[color=Bittersweet] (4,0) arc (0:180:1 and 0.7);
        \draw[color=Bittersweet] (6,0) arc (0:180:1 and 0.7);
        \draw[color=CadetBlue] (4,0) arc (180:360:1 and 0.7);
        \draw[color=CadetBlue] (2,0) arc (180:360:1 and 0.7);
        \draw[color=CadetBlue] (0,0) arc (180:360:1 and 0.7);
        \begin{scope}
        \clip(-1,-0.7)rectangle(0,0);
        \draw[color=CadetBlue, dashed] (-2,0) arc (180:360:1 and 0.7);
        \end{scope}
        \begin{scope}
        \clip(-0.3,-0.7)rectangle(0,0);
        \draw[color=CadetBlue] (-2,0) arc (180:360:1 and 0.7);
        \end{scope}
        \begin{scope}
        \clip(-0.32,0.7)rectangle(0,0);
        \draw[color=Bittersweet] (0,0) arc (0:180:1 and 0.7);
        \end{scope}
        \begin{scope}
        \clip(-1,0.7)rectangle(0,0);
        \draw[color=Bittersweet, dashed] (0,0) arc (0:180:1 and 0.7);
        \end{scope}

        \draw[color=Bittersweet, -Stealth] (0.99,0.7) -- (1.01,0.7);
        \draw[color=Bittersweet, -Stealth] (2.99,0.7) -- (3.01,0.7);
        \draw[color=Bittersweet, -Stealth] (4.99,0.7) -- (5.01,0.7) node[anchor=south] {$g_1$};
        \draw[color=CadetBlue, -{Triangle[open]}] (0.99,-0.7) -- (1.01,-0.7);
        \draw[color=CadetBlue, -{Triangle[open]}] (2.99,-0.7) -- (3.01,-0.7);
        \draw[color=CadetBlue, -{Triangle[open]}] (4.99,-0.7) -- (5.01,-0.7) node[anchor=north] {$g_2$};
        
        \draw[color=CadetBlue] (7.56,1.34) circle (0.5);
        \draw[color=CadetBlue]  (9.349,2.235) circle (0.5);
        \draw[color=CadetBlue] (11.138,3.129) circle (0.5);
        \draw[color=Bittersweet] (7.56,-1.34) circle (0.5);
        \draw[color=Bittersweet]  (9.349,-2.235) circle (0.5);
        \draw[color=Bittersweet] (11.138,-3.129) circle (0.5);

        \draw[color=CadetBlue, -{Triangle[open]}] (7.55,1.8387) -- (7.56,1.84);
        \draw[color=CadetBlue, -{Triangle[open]}] (9.339,2.7337) -- (9.349,2.735);
        \draw[color=CadetBlue, -{Triangle[open]}] (11.128,3.6277) -- (11.138,3.629);
        \draw[color=Bittersweet, -Stealth] (7.55,-1.8387) -- (7.56,-1.84);
        \draw[color=Bittersweet, -Stealth] (9.339,-2.7337) -- (9.349,-2.735);
        \draw[color=Bittersweet, -Stealth] (11.128,-3.6277) -- (11.138,-3.629);

        \filldraw[black] (6,0) circle (1.5pt) node[anchor=east]{\small$(1,1)$};
        \filldraw[black] (4,0) circle (1.5pt) node[anchor=east]{\small$(1,2)$};
        \filldraw[black] (2,0) circle (1.5pt) node[anchor=east]{\small$(1,3)$};
        \filldraw[black] (0,0) circle (1.5pt) node[anchor=east]{\small$(1,4)$};
        \filldraw[black] (7.789,0.894) circle (1.5pt); \node[anchor=west] at (7.6,0.6) {\small$(2,1)$};
        \filldraw[black] (9.578,1.789) circle (1.5pt); \node[anchor=west] at (9.4,1.5) {\small$(2,2)$};
        \filldraw[black] (11.367,2.683) circle (1.5pt); 
        \node[anchor=west] at (11.2,2.4){\small$(2,3)$};
        \filldraw[black] (7.789,-0.894) circle (1.5pt); 
        \node[anchor=west] at (7.6,-0.6) {\small$(3,1)$};
        \filldraw[black] (9.578,-1.789) circle (1.5pt); 
        \node[anchor=west] at (9.4,-1.5) {\small$(3,2)$};
        \filldraw[black] (11.367,-2.683) circle (1.5pt); 
        \node[anchor=west] at (11.2,-2.4) {\small$(3,3)$};
    \end{tikzpicture}
    \caption{Schreier graph of the action of $\mathcal{G}_3$ on $X_3$.}
    \label{H3}
\end{figure}

Brown showed in \cite[Theorem~5.1]{Brown} that each $\mathcal{G}_n$ is of type $FP_{n-1}$ but not of type $FP_n$ for all $n$ and that $\mathcal{G}_n$ is finitely presented for all $n \geq 3$. Recall that, for any $n$, if a group $G$ is of type $FP_n$ and is finitely presented then $G$ is of type $F_n$ (see e.g. \cite[Section~8.7]{Brown94}). 
This allows us to construct, for any $n \geq 2$, a pair of groups $H \leq G$ such that both $G$ and $H$ are of type $F_n$ and the quotient space $H \bs G$ is a quasi-line but $H$ is not a virtual fiber subgroup.

\begin{example} \label{Fn example}
    Fix $n \geq 3$ and let $\sigma$ be a permutation of $\{1,\dots,n\}$ such that $\sigma(1) = 1$.
    Define $\alpha: X_n \rightarrow X_n$ by $\alpha(i,m) = (\sigma(i),m)$
    and let $L$ be the subgroup of the permutation group of $X_n$ generated by $S_n \cup \{\alpha\}$. Note that $\alpha \mathcal{G}_n \alpha^{-1} = \mathcal{G}_n$ and $\alpha \notin \mathcal{G}_n$ unless $\sigma = \id$. Therefore $L = \mathcal{G}_n \rtimes \mathbb{Z}/ m \mathbb{Z}$, where $m$ is the order of $\sigma$, and $L$ is of type $F_{n-1}$ by \cite[Corollary~9]{Alonso}. Let $M \coloneqq \Stab_L((1,1))$. Then $\alpha \in M$ so $M = H_n \rtimes \la \alpha \ra$ and $M$ also has type $F_{n-1}$.\footnote{In fact, $M \cong L$: Let $\psi:X_n \rightarrow X_n - \{(1,1)\}$ be the bijection defined by $\psi(i,n) = (i,n)$ if $i \neq 1$ and $\psi(i,n) = (i,n+1)$ otherwise. Define $\overline{\psi}:\mathcal{G}_n \rightarrow H_n$ as follows. If $f \in \mathcal{G}_n$ then $\overline{\psi}(f)((1,1)) = (1,1)$ and $\overline{\psi}(f)_{|X_n - \{(1,1)\}} = \psi \circ f \circ \psi^{-1}$. One can check that $\overline{\psi}$ is an $\la\alpha\ra$-equivariant isomorphism so it induces an isomorphism $L \rightarrow M$.}
    
    Moreover, the Schreier graph of $M$ with respect to $S_n \cup \{\alpha\}$ is a quasi-tree whose number of ends is the number of orbits of the action of $\la\sigma \ra$ on $\{1, \dots, n\}$ and $L=M \cup Mg_1M$. By Theorem~\ref{main}, this implies that $M$ is not a virtual fiber subgroup.
    
    Using cyclic notation, if we take for instance $\sigma = (2,\dots,n)$ then this provides an example of a pair $H \leq G$ where $H \bs G$ is a quasi-line, both $H$ and $G$ are of type $F_{n-1}$ and $H$ is not a virtual fiber subgroup.
\end{example}

Suppose that $G,N,Q$ are finitely generated groups which fit into a short exact sequence:
\[
\begin{tikzcd}
    1 \arrow[r] &N \arrow[r] &G \arrow[r,"\pi"] &Q \arrow[r] & 1.
\end{tikzcd}
\]
Fix a (set-theoretic) section $\varphi: Q \rightarrow G$ of $\pi$. If $K \leq Q$ is a finitely generated subgroup such that $K \bs Q$ is an $n$-ended quasi-tree for some $2 \leq n < \infty$ then the pull-back $H \coloneqq \pi^{-1}(K)$ is a finitely generated subgroup of $G$ and $H \bs G$ is an $n$-ended quasi-tree. Moreover, if there is a finite subset $F \subseteq Q$ such that $Q =  KFK$ then $G = H\varphi(F)H$. Therefore if $K \leq Q$ is not a virtual fiber then neither is $H \leq G$. As demonstrated by Example~\ref{Fn example}, there exists, for each $2 \leq n < \infty$, such a pair $K \leq Q$ where $Q$ is finitely presented. Therefore one can use the Rips construction \cite{Rips} or any of its variants (see e.g. \cite{Arenas} and the references therein), to produce examples of finitely generated pairs $H \leq G$ where $H \bs G$ is an $n$-ended quasi-tree, $H$ is not a virtual fiber and $G$ is, in some precise way, ``well-behaved". For example, Arenas' construction \cite{Arenas} yields a group $G$ which is torsion-free hyperbolic (in particular, of type $F$) and virtually compact special.

\bibliographystyle{alpha}
\bibliography{Biblio}

\bigskip
{\footnotesize
  \noindent
  {\textsc{University of Bristol, School of Mathematics, Bristol, UK}} \par\nopagebreak
  \texttt{penelope.azuelos@bristol.ac.uk}

\end{document}